\documentclass[]{article}

\usepackage{amsmath,amssymb,exscale,epsfig}
\usepackage{delarray}
\usepackage{amsmath}    % les symboles les plus frM-iquents
\usepackage{amssymb, bbm}    % des symboles
\usepackage{amsfonts}   % des fontes, eg pour \mathbb
\usepackage{graphics}
\usepackage{graphicx}
\usepackage{stmaryrd}

\usepackage{amsopn}
\usepackage[ansinew]{inputenc}
\usepackage{tabularx}
\usepackage{xcolor}
\usepackage{natbib}

\usepackage{authblk}

\newtheorem{theoreme}{Theorem}

\newtheorem{lemme}{Lemma}
\newtheorem{definition}{Definition}

\newtheorem{corollaire}{Corollary}

\newcommand{\EE}{\ensuremath{{\mathrm{\bf E}}}}

\newcommand{\A}{\ensuremath{{\mathcal{A}}}}
\newcommand{\LL}{\ensuremath{{\mathcal{L}}}}

\newcommand{\G}{\ensuremath{\mathcal{G}}}

\renewcommand{\P}{\ensuremath{{\mathcal{P}}}}
\newcommand{\I}{\ensuremath{{\mathcal{I}}}}

\newcommand{\F}{\ensuremath{{\mathcal{F}}}}

\newcommand{\R}{\ensuremath{{\mathbb{R}}}}

\newcommand{\Rd}{\ensuremath{{\mathbb{R}^d}}}

\newcommand{\Z}{\ensuremath{\mathbb{Z}}}

\newcommand{\om}{\ensuremath{\omega}}

\renewcommand{\O}{\ensuremath{\Omega}}

\newcommand{\1}{\ensuremath{\mbox{\rm 1\kern-0.23em I}}}
\renewcommand{\L}{\ensuremath{\Lambda}}
\newcommand{\Ln}{\ensuremath{{\Lambda_n}}}
\newcommand{\Lno}{\ensuremath{{\Lambda_n^\oplus}}}
\renewcommand{\l}{\ensuremath{\lambda}}

\title{\sc Variational principle for Gibbs point processes with finite range interaction}

\author[1]{David Dereudre}
\affil[1]{Laboratoire de Math\'ematiques Paul Painlev\'e\\University of Lille 1, France}

\begin{document}
\maketitle

\begin{abstract}

The variational principle for Gibbs point processes with general finite range interaction is proved. Namely, the Gibbs point processes are identified as the minimizers of the free excess energy equals to the sum of the specific entropy and the mean energy. The interaction is very general and includes superstable pairwise potential, finite or infinite multibody potential, geometrical interaction, hardcore interaction. The only restrictive assumption involves the finite range property.

  \bigskip


\end{abstract}

\section{Introduction}

Gibbs point processes are popular models to describe the repartition of points or geometrical  structures in space. 
They appeared first for modelling continuum interacting particles in statistical mechanics. Now they are widely used in as different domains as astronomy, biology, computer science, ecology, forestry,  image analysis, materials science. The main reason is that they provide a clear interpretation of the interactions between the points, such as attraction or repulsion depending on their interdistance. We refer to \cite{B-Pre76}, \cite{chiu2013},  and \cite{vanlieshout2000} for 
classical text books on Gibbs point processes, including examples and applications.

The Gibbs point processes are defined via their local unnormalized conditional densities of the form $e^{-H}$ where $H$ is an energy functional. They are the equilibrium states of the DLR equations (see definition 2). A variational principle, coming from the statistical physics, claims that the Gibbs measures are also the minimizers of the free excess energy equals to  the entropy plus the mean energy. More precisely, for any stationary probability measure $P$ on the space of configurations in $\Rd$, the specific entropy $\I(P)$ with respect to the Poisson point process  and the mean energy per unit volume $H(P)$ are defined via thermodynamic limits (see  \eqref{entropy} and  \eqref{meanenergy}). The principle claims that the Gibbs measures are exactly the probability measures $P$ which minimize the functional $P\mapsto \I(P)+H(P)$. It thus supports the common belief that the Gibbs measures provide a proper description of physical systems in thermodynamic equilibrium. They are many applications of the variational principle in physics and mathematics. In statistical mechanics, the phase transition phenomenon (non uniqueness of Gibbs measures) can be proved in studying the geometry of the set of Gibbs measures and in particular its extremal points. The variational principle is a key tool in this study (see \cite{georgii} for a general presentation). In probability theory, it is related to the large deviation principle for the empirical field \cite{Georgii94}. In spatial statistic, it is a crucial identifiability assumption for the consistency of the maximum likelihood estimator \cite{DereudreLavancier15}. This last recent paper highlights the importance of the variational principle  for models coming from the spatial statistic. It was our initial motivation for the present paper.

For the lattice Gibbs models, the variational principle is well established and a general proof can be found in \cite{B-Pre76}, Section 7. The first proof was for the Ising model in \cite{landfordRuelle}. In the setting of Gibbs point processes, they are less results. In  \cite{Georgii94, Georgii94b} the author proves the variational principle for the pairwise potential energy  $H(\omega)=\sum_{x,y} \phi(|x-y|)$ where the sum is over all couples $\{x,y\}$ in the configuration $\om$.
The potential $\phi$ is assumed to be non-integrably divergent at the origin (i.e. $\int_0 ^1 \phi(r)r^{d-1}ds =\infty$) producing a strong repulsion when the particles are closed to each other. A typical examples is the Lennard Jones pairwise potential $\phi(r)=ar^{-12}-br^{-6}$. In another work the variational principle is proved for the Delaunay-tile interaction \cite{A-DerGeo09}. The energy function has the following form $H(\omega)=\sum_{T} \phi(T)$ where the sum is over all triangles $T$ of the Delaunay triangulation based on $\omega$. It is a continuum spatial version of nearest neighbours interaction models. As fas as we know both papers are the only ones proving the variational principle for  Gibbs point processes models. Unfortunately many interesting energy functions are not covered by these results, as for example any pairwise potential energy with bounded potential $\phi$. In particular, the well-studied Strauss model in spatial statistics with the pairwise potential $\phi(r)=\1_{[0,R]}(r)$ is uncovered. In stochastic geometry, the Area-interaction or the Quermass-interaction are not covered as well (see \cite{A-Baddeley95} and \cite{A-Kendall99}).

In this paper we prove the variational principle for Gibbs point processes with general finite range interaction. The other assumptions are standard and satisfied by all models we met in statistical mechanics and spatial statistic (see Section 3). In particular, our setting includes superstable pairwise potential, finite or infinite multibody potential, geometrical interaction, hardcore interaction. The proof is based on fine controls of the relative entropy of $P_\L$ with respect to the Gibbs measure on $\L$, where $P$ is any stationary field on $\Rd$ and $\L$ an observable window tending to $\R^d$. This strategy was already present in \cite{B-Pre76}, \cite{Georgii94b} and \cite{Georgii94}.

The paper is organized as follows. In section 2, we introduce the notations and the Gibbs models. The variational principle and the main theorem are presented in Section 3. Two standard examples are given in Section 4; the  superstable pairwise potential with compact support and the Quermass interaction. Section 4 is devoted to the proof of our main theorem.

\section{The Gibbs models}

\subsection{State spaces and reference measures}

Our setting is the Euclidean space $\Rd$ of arbitrary dimension $d \ge 1$ equipped with its Borel $\sigma$-field. An element of $\Rd$ is denoted by $x$ and the Lebesgue measure on $\Rd$ is denoted by $\l^d$. A {\bf configuration} is a subset $\om$ of $\Rd$ 
which is locally finite, meaning that $\om\cap\L$ has finite cardinality
$N_\L(\om)=\#(\om\cap\L)$ for every bounded Borel set $\L$. 
The space $\O$ of all configurations is equipped with the 
$\sigma$-algebra $\mathcal{F}$ generated by the counting variables $N_{\L}$. The space of finite configurations is denoted by $\O_f$.

The symbol $\L$ will always refer to a bounded Borel set in $\R^d$.
It will often be convenient to write $\om_\L$ in place of $\om\cap\L$. 
We abbreviate $\om\cup\{x\}$ to $\om\cup x$ 
and abbreviate  $\om\backslash\{x\}$ to $\om\backslash x$
for every $\om$ and every $x$ in $\om$.  

As usual, we take the reference measure on $(\O,\mathcal{F})$ 
to be the distribution $\pi$ of the Poisson point process 
with intensity measure $\l^d$  on $\Rd$.
Recall that $\pi$ is the unique probability measure
on $(\O,\mathcal{F})$ such that the following hold for all subsets $\L$:
(i) $N_{\L}$ is Poisson distributed with parameter $\l^d(\L)$, 
and (ii) conditional on $N_\L=n$, 
the $n$ points in $\L$ are independent with uniform distribution on $\L$. 
The Poisson point process restricted to $\L$ will be denoted $\pi_\L$. 

Translation by a vector $u \in \Rd$ is denoted by $\tau_u$,
either acting on $\Rd$ or on $\O$. 
A probability $P$ on $\O$ is said stationary if  $P=P\circ\tau_u^{-1}$ for any $u$ in $\Rd$. In this paper we consider only stationary probability measures $P$ with finite finite intensity measure (i.e. $\EE_P(N_{[0,1]^d})<+\infty$). We denote by $\P$ the space of such probability measures. 

%We denote by $\Delta_0$, $\I_n$ and $\L_n$ the following sets 
%\begin{equation}\label{sets} \Delta_0=[0,1[^d, \quad \I_n=\{-n,-n+1,\ldots, n-1\}^d \quad \text{ and } \L_n=\bigcup_{k\in\I_n} \tau_k(\Delta_0)=[-n,n[^d.\end{equation}

%
%In the following, some infinite range interaction processes will be considered. To ensure their existence, we must restrict the set of configurations to the so-called set of tempered configurations as in \cite{Ruelle70}. The definition of the latter may depend on the type of interactions at hand. In case of superstable pairwise interactions, it is simply defined by
%$$\Omega_T=\{\omega\in\Omega;\; \exists t>0, \forall n\geq 1, \sum_{i\in \I_n} N^2_{ \tau_i(\Delta_0)}(\omega) \leq t (2n)^d\}.$$
%% $\O_T=\cup_{t>0} \O(t)$. 

\subsection{Gibbs point processes models}

We consider a measurable function $H$ from $\Omega_f$ to $\R\cup\{+\infty\}$ which is called energy function. We assume that $H$ is {\bf stationary}; for any $\om\in\Omega_f$ and any $u\in\Rd$, $H(\om)=H(\tau_u(\om)$. We assume also that the energy function $H$ is {\bf hereditary} which means that for any $x$ in $\Rd$ and $\om$ in $\Omega_f$, $H(\om\cup\{x\})=+\infty$ as soon as $H(\om)=+\infty$. The energy $H$ is said {\bf non-degenerate} if $H(\{0\})\neq \infty$ and  $H(\emptyset)=0$.  We assume also that $H$ is {\bf stable} which means that there exists $A>0$ such that for any finite configuration $\om\in\O_f$

\begin{equation}\label{stable}
H(\om)\ge -A N(\om).
\end{equation}

All these assumptions are standard and non restrictive. The main restriction in the present paper is the {\bf finite range} assumption  which means that there exists $R\ge 0$ such that for any configuration $\omega$, any bounded set $\Lambda$ the quantity

\begin{equation}\label{EnergyL}
H_\L(\om):= H(\om_{\Lambda'})-H(\om_{\Lambda'\backslash \Lambda})
\end{equation}
(with the convention $\infty-\infty=0$) does not depend on the choice of $\L'$ as soon as $\L\oplus B(0,R) \subset \L'$. $H_\L(\om)$ represents the energy of $\om_\L$  inside $\L$ given 
the configuration $\om_{\L^c}$ outside $\L$. 
%
%We shall give a particular emphasis in Section~\ref{pair interaction} to pair potential interactions, that are by far the most widely used models in practice.  {\bf Pair potential interactions} take the particular form :
%\begin{equation}\label{defH}
%H (\om)=z\ n(\om) + \sum_{\{x,y\}\subset \om} \phi(x-y),\end{equation}
%where $z>0$ is the intensity parameter and $\phi$ is a symetric function from $\R^d$ to  $\R\cup\{+\infty\}$ with a compact support.\\
%
%

The Gibbs measures $P$ associated to $H$ are defined through their local conditional specification, as described below. We denote by $\Omega_\infty$ the set of configurations $\om\in\Omega$ such that for any $\Lambda$, $H(\omega_\L)<+\infty$. So for every $\L$ and every configuration $\om\in\Omega_\infty$, the local conditional density  $f_\L$ of $P$ with respect to $\pi_\L$ is defined by
\begin{eqnarray}
\label{localdensity}
f_\L(\om) & = & \frac{1}{Z_\L(\om_{\L^c})} e^{-H_\L (\om)},
\end{eqnarray} 
where $ Z_\L(\om_{\L^c})$ is the normalization constant given by 
$$Z_\L(\om_{\L^c})= \int e^{-H_\L(\om'_\L \cup \om_{\L^c})}\pi_\L(d\om'_\L).$$
Let us note that $ 0< Z_\L(\om_{\L^c})< +\infty$ since $H$ is stable and non-degenerate.

We are now in position to define the Gibbs measures associated to $H$ (See \cite{B-Pre76} for instance).

\begin{definition}
A probability measure $P$ on $\O$ is a {\bf Gibbs measure}
for the energy function $H$  
if $P(\Omega_\infty)=1$ and if for every bounded borel set $\L$, for any measurable and bounded function $g$ from $\O$ to $\R$,
\begin{equation}\label{DLR}
\int g(\om)P(d\om) = \int \int g(\om'_\L \cup \om_{\L^c}) f_\L(\om'_\L \cup \om_{\L^c}) \pi_\L({\rm d}\om'_\L) P(d\om).
\end{equation} 
Equivalently, 
for $P$-almost every $\om$ the conditional law of $P$ given $\om_{\L^c}$ 
is absolutely continuous with respect to $\pi_\L$ with the density $f_\L$ defined in \eqref{localdensity}.
\end{definition}

The equations (\ref{DLR}) are called the Dobrushin--Lanford--Ruelle (DLR)
equations. The existence of such Gibbs measures, in the present setting of finite range stable interactions, is done in \cite{A-DerDroGeo09}, Corollary 3.4 and Remark 3.1. Note that the uniqueness of such $P$ does not necessarily hold, leading to the phase transition phenomenon. We denote by $\G_H$ the set of all Gibbs measures for the energy $H$. 

\section{Variational Principle}

The variational principle in statistical mechanics claims that the Gibbs measures are the minimizers of the free excess energy defined by the sum of the the mean energy and the specific entropy. Moreover the minimum is equal to minus the pressure. Let us first define precisely all these macroscopic  quantities. For the sake of simplicity we consider the macroscopic limit along the sequence of sets $\L_n=[-n,n]^d$, $n\ge 1$. Limits $\L\to\Rd$ in the Van-Hove sens could have been considered as well.

Let $P$ be a stationary probability measure in $\P$. The {\bf specific entropy} of $P$ is defined as the limit

\begin{equation}\label{entropy}
 \I(P)=\lim_{n\to +\infty} \frac{1}{|\L_n|} \I(P_\Ln,\pi_\Ln),
 \end{equation}
where for any probability measures $\mu$ and $\nu$
$$ \I(\mu,\nu)= \left\{
 \begin{array}{ll}
\int \ln(f) d\mu  & \text{ if } \mu\ll\nu \text{ with density } f\\
+\infty & \text{otherwise}
\end{array}\right..$$

 Note that the limit in $\eqref{entropy}$ always exists; see \cite{georgii} for general results on specific entropy.

Let us now introduce the {\bf Pressure}. It is defined as the following limit.

\begin{equation}\label{pressure}
p_H:=\lim_{n\to +\infty} \frac{1}{|\L_n|} \ln (Z_{n}),
\end{equation}

where $Z_n=Z_{\L_n}(\emptyset)$  is the partition function with empty boundary condition.

In the following lemma we show that $p_H$ always exists in the setting of the present paper.

\begin{lemme}\label{existencepressure}
Assuming that the energy function $H$ is finite range, stable and non-degenerate, then the pressure $p_H$ defined in \eqref{pressure} exists and belongs to $[-1,(e^A-1)]$.
\end{lemme}

\begin{proof}

For any set $\L$ we denote by $\L^\ominus$ the set 
$$ \L^\ominus=\{ x\in\L, B(x,R_0)\subset \L\},$$

where $R_0$ is an integer larger than the range of the interaction $R$. So for $n>R_0$, $\L_n^\ominus=\L_{n-R_0}$.

 For any $R_0<m<n$, we consider the Euclidean division $n=km+l$ with $0\le l<m$, $k \ge 0$. Let $(\L_m^i)_{1\le i \le k^d}$ a family of $k^d$ disjoint cubes inside $\L_n$ where each cube is a translation of $\L_m$. 

From the definition of the partition function

\begin{eqnarray*}
Z_n & \ge & \pi_{\L_n}(\om_{\L_n\backslash \cup_i \L_m^{i, \ominus}}=\emptyset)\int e^{-H(\om)}\pi_{\cup_i \L_m^{i, \ominus}}(d\om)  \\
& = & e^{-(|\L_n|-k^d|\L^\ominus_m|)} \prod_i Z_{\L_m^{i,\ominus}}(\emptyset)\\
& \ge & e^{-( k^d2dR_0(2m)^{d-1}+2dm(2n)^{d-1}) } Z_{{m-R_0}}^{k^d}.
\end{eqnarray*}
So since $|\L_n|/k^d$ goes to $|\L_m|$ when $n$ goes to infinity, 

$$ \liminf_{n\to \infty} \frac{1}{|\L_n|} \ln(Z_n) \ge \frac{1}{|\L_m|} \big( Z_{m-R_0}- 2dR_0 (2m)^{d-1}  \big).$$

This inequality holds for each $m\ge R_0$. So, letting $m$ tends to infinity

$$ \liminf_{n\to \infty} \frac{1}{|\L_n|} \ln(Z_n) \ge \limsup_{m\to \infty} \frac{1}{|\L_m|}  Z_{m-R_0} = \limsup_{m\to \infty} \frac{1}{|\L_m|}  Z_{m}$$

which proves that the limit exists in $\R\cup\{\pm \infty\}$.

Thanks to the stability and the non degeneracy of $H$ we get that
$$  e^{-|\L_n|} \le Z_n \le e^{|\L_n|(e^A-1)}$$
which implies that $p_H\in[-1,(e^A-1)]$.

\end{proof}

The last macroscopic quantity involves the mean energy of a stationary probability measure $P$. It is also defined by a limit but, in opposition to the other macroscopic quantities, we have to assume that it exists. The proof of such existence is based on stationary arguments and a nice representation of the energy contribution per unit volume. Examples are given in Section \ref{Sec:Examples}. So for any stationary probability measure $P$ we assume that the following limit exists

\begin{equation}\label{meanenergy}
H(P):=\lim_{n\to \infty} \frac{1}{|\L_n|} \int H(\om_{\L_n}) P(d\om).
\end{equation}
and we call the limit {\bf mean energy} of $P$.

We need to introduce a last technical assumption on the boundary effects of $H$. We assume that for any $P$ in $\G_H$ 

\begin{equation}\label{boundary}
\lim_{n\to\infty} \frac{1}{|\L_n|} \int \partial H_{\L_n}(\om)P(d\om)=0,
\end{equation}
where $\partial H_{\L_n}(\om)= H_{\L_n}(\om)-H(\om_{\L_n})$.  This assumption is satisfied by all the examples we meet.

\begin{theoreme}\label{mainTh}

 We assume that $H$ is stationary, hereditary, non-degenerate, stable and finite range. Moreover we assume that the mean energy exists for any stationary probability measure $P$ (i.e. the limit \eqref{meanenergy} exists) and that the boundary effects assumption \eqref{boundary} holds. Then for any stationary probability measure $P\in\P$

\begin{equation}\label{inequality}
 I(P)+H(P) \ge -p_H,
\end{equation}
 
with equality if and only if $P$ is a Gibbs measure (i.e. $P\in\G_H$).

\end{theoreme}

\section{Examples}\label{Sec:Examples}

In this section we present two examples of energy functions included in the setting of Theorem \ref{mainTh}. The first example is the standard superstable pairwise potential energy. The second example involves the Quermass interaction which is an energy function for morphological patterns built by unions of random convex sets. It can be also viewed as a infinite body potential interaction. The main restriction in Theorem \ref{mainTh} is the finite range property and so any standard examples, having this property, could have been considered as well.

\subsection{Pairwise potential}

In this section the energy function has the following expression: for any finite configuration $\om\in\Omega_f$

\begin{equation}\label{pairwise}
H(\omega)=zN(\omega)+\sum_{\{x,y\}\subset \omega} \phi(x-y),
\end{equation}

where $\phi$ is a symmetric function from $\Rd$ to $\R\cup\{+\infty\}$ with compact support. The parameter $z>0$ is called activity and allow to change the intensity of the reference Poisson point process. The potential $\phi$ is said stable if the associated energy $H$ in \eqref{pairwise} is stable. In the following we need that the potential $\phi$ is {\bf superstable} which means that $\phi$ is the sum of stable potential and a positive potential which is non negative around the origin. See \cite{Ruelle} for examples of stable and superstable pairwise potentials. In this setting the variational principle holds as a corollary of Theorem \ref{mainTh}.

\begin{corollaire}
Let $H$ be a energy function coming from a superstable pairwise potential $\phi$ given by \eqref{pairwise}. Then for any stationary probability measure $P\in\P$ 

\begin{equation}
 I(P)+H(P) \ge -p_H,
\end{equation}
 
with equality if and only if $P$ is a Gibbs measure (i.e. $P\in\G_H$). The expression of $H(P)$ is given in \eqref{MeanEPW}.
\end{corollaire}

\begin{proof}

Let us check the assumptions of Theorem \ref{mainTh}. It is obvious that $H$ is stationary, hereditary, non degenerate, stable and finite range. The existence of the mean energy is proved in \cite{Georgii94}, Theorem 1. It is given by

\begin{equation}\label{MeanEPW}
 H(P)=\left\{ \begin{array}{ll}
\frac 12\int \sum_{0\neq x\in\om} \phi(x)P^0(d\om) & \text{ if  } E_P(N^2_{[0,1]^d})<\infty \\
+\infty & \text{ otherwise}
\end{array}\right.
\end{equation}

where $P^0$ is the Palm measure of $P$. Recall that $P^0$ can be viewed as the natural version of the conditional probability $P(.|0\in\om)$ (see \cite{MKM78} for more details). So, it remains to prove the boundary assumption \eqref{boundary}. Let $P$ a Gibbs measure in $\G_H$. A simple computation gives that for any $\om\in\Omega$

$$ \partial H_{\Ln(\om)}= \sum_{x\in\om_{\Lno\backslash \Ln}}
\sum_{ y\in\om_{\Ln\backslash \L_n^\ominus}} \phi(x-y),$$

where $\Lno=\L_{n+R_0}$ and $\L_n^\ominus=\L_{n-R_0}$ with $R_0$ an integer larger than the range of the interaction $R$. Using the GNZ equation (see \cite{A-NguZes79b}), the stationarity of $P$ we obtain

\begin{eqnarray*}
|E_P(\partial H_{\Ln})| & \le & \int\sum_{x\in\om_{\Lno\backslash \Ln}} \sum_{y\in\om\backslash x} |\phi(x-y)|P(d\om)\\
 & = & \int\int_{\Lno\backslash \Ln} e^{-z-\sum_{y\in\om} \phi(x-y)}  \sum_{y\in\om} |\phi(x-y)|dx P(d\om)\\
 & = &|\Lno\backslash \Ln|e^{-z} \int e^{-\sum_{y\in\om_{B(0,R_0)}} \phi(y)}  \sum_{y\in\om_{B(0,R_0)}} |\phi(y)| P(d\om).\\
\end{eqnarray*}

Since $\phi$ is stable we deduce that $\phi\ge -A-2z$. So denoting by $C:=\sup_{c\in[-A-2z;+\infty)} |c|e^{-c}<\infty$ we find that

\begin{eqnarray}\label{bord}
|E_P(\partial H_{\Ln})| 
 & \le & |\Lno\backslash \Ln|Ce^{-z} \int N_{B(0,R_0)}(\om)e^{(A+2z)N_{B(0,R_0)}(\om)} P(d\om).
\end{eqnarray}
Using the estimates in \cite{Ruelle70} corollary 2.9, the integral in the right term of \eqref{bord} is finite. The boundary assumption \eqref{boundary} follows.

\end{proof}

\subsection{Quermass interaction}

The Quermass process is a morphological interacting model introduced in \cite{A-Kendall99} which is a generalization of the well-known Widom-Rowlinson process or Area Process (see \cite{Widom70}, \cite{A-Baddeley95}). Since the existence of the Quermass process is only proved in the case $d\le 2$  we restrict the following to the non trivial case $d=2$. For any finite configuration $\om$, $L(\om)$ denotes the set $\cup_{x\in\omega} B(x,r)$ and the energy is defined as a linear combination of the Minkowski functionals;
\begin{equation}\label{Quermassinteraction}
 H(\om)= \theta_1 \A \Big(L(\om)\Big) + \theta_2 \LL \Big(L(\om)\Big) +\theta_3 \chi \Big(L(\om)\Big),
 \end{equation}
where $r>0$, $\theta_i\in\R$ ,  $i=1\ldots 3$ are parameters and $\A$, $\LL$, $\chi$ are respectively the area, the perimeter and  the Euler-Poincar\'e characteristic functionals. Recall that $\chi(L(\om))$ is equal to $ N_{cc}(L(\om))-N_h(L(\om))$ where $N_{cc}(L(\om))$ denotes the number of connected components in $L(\om)$ and $N_h(L(\om))$ the number of holes. We refer to \cite{chiu2013} for more details about Minkowski functionals.

\begin{corollaire} Let $H$ be the Quermass interaction given in \eqref{Quermassinteraction}. Then for any stationary probability measure $P\in\P$ 

\begin{equation}
 I(P)+H(P) \ge -p_H,
\end{equation}
 
with equality if and only if $P$ is a Gibbs measure (i.e. $P\in\G_H$). The expression of $H(P)$ is given in \eqref{MeanEQ}. 
\end{corollaire}

\begin{proof}

As in the previous section, we check the assumptions of Theorem \ref{mainTh}. It is obvious that $H$ is stationary, hereditary and non degenerate. In dimension $d=2$ the functional $\chi$ satisfied the following bound

\begin{equation}\label{boundchi}
|\chi(L(\om))| \le 3N(\om),
\end{equation}
see \cite{A-Kendall99}. The stability of $H$ follows easily. The finite range assumption is a consequence of the additivity of Minkowski functionals. Note that the range of the interaction is $R=2r$. In the following we denote by $C$ the cube $[0,1]^2$, by $\partial C$ the boundary of $C$ and by $\hat C$ the double edges $\{0\}\times [0,1] \cup [0,1]\times \{0\}$. For any  $k\in\Z^2$ we consider also the translations $C_k=\tau_k (C)$, $\partial C_k=\tau_k (\partial C)$ and $\hat C_k=\tau_k(\hat C)$. Thanks to the additivity of Minkowski functionals we obtain that for any finite configuration $\omega$ and any $n\ge 1$ 

 \begin{eqnarray}\label{decomposition}
H(\omega_\Ln) &=& \sum_{k\in\{-n, n-1\}^2} \bigg(\theta_1 \A\big(L(\om)\cap C_k\big) +  \theta_2 \Big[\LL\big(L(\om)\cap C_k\big)- \LL\big(L(\om)\cap \partial C_k\big)\Big] \nonumber  \\
& & \qquad \qquad +\theta_3 \Big[ \chi\big(L(\om)\cap C_k\big) - N_{cc}\big(L(\om)\cap \hat C_k\big)\Big]\bigg) + R_n(\om_{\Ln}),
\end{eqnarray}
which gives the energy contribution of each cube $C_k$ in  $H(\omega_\Ln)$. Thanks to \eqref{boundchi} and obvious bounds for $\A$ and $\LL$, the boundary term $R_n(\om_{\Ln})$ satisfies for some constant $c>0$

\begin{equation}\label{boun}
 | R_n(\om_{\Ln})| \le c N\left(\om_{\Ln\backslash \L_{n-R_0}}\right).
 \end{equation}

For any stationary probability measure $P\in\P$ we deduce easily from \eqref{decomposition} and  \eqref{boun}  the existence of the mean energy \eqref{meanenergy} with
\begin{equation}\label{MeanEQ}
 H(P)= \int \theta_1 \A\big(L(\om)\cap C\big) +  \theta_2 \Big[\LL\big(L(\om)\cap C\big)- \LL\big(L(\om)\cap \partial C\big)\Big] +\theta_3 \Big[ \chi\big(L(\om)\cap C\big) - N_{cc}\big(L(\om)\cap \hat C\big)\Big]P(d\om).
 \end{equation}

Thanks to \eqref{decomposition} and \eqref{boun}, the boundary assumption \eqref{boundary} is satisfied as well.

\end{proof}

\section{ Proof of Theorem \ref{mainTh}}

Let us start by proving the inequality \eqref{inequality} for any stationary probability measure $P\in\P$. For $n\ge1$ we define the Gibbs measure on $\L_n$ with free boundary condition by 

\begin{equation}\label{FBGibbs}
Q_n(d\om_{\L_n})=\frac{1}{Z_n}e^{-H(\om_{\Ln})}\pi_{\L_n}(d\om_{\L_n}).
\end{equation}

So

\begin{eqnarray*}
\I(P_\Ln,Q_n) & = &\int \ln\left( \frac{dP_\Ln}{dQ_n}(\om_\Ln)\right)dP_\Ln(\om_\Ln)\nonumber\\
& = &\int \ln\left( \frac{dP_\Ln}{d\pi_\Ln}(\om_\Ln) \frac{d\pi_\Ln}{dQ_n}(\om_\Ln)\right)dP_\Ln(\om_\Ln)\nonumber\\
& = &\I(P_\Ln,\pi_\Ln)+\int H(\om_\Ln)dP(\om)+\ln(Z_n),
\end{eqnarray*}

which implies that

\begin{equation}\label{representation}
\lim_{n\to\infty} \frac{1}{|\Ln|}\I_{\Ln}(P_\Ln,Q_n)=\I(P)+H(P)+p_H.
\end{equation}

Since $\I_{\Ln}(P_\Ln,Q_n)$ is positive the inequality \eqref{inequality} follows. Let us now prove that for any $P\in\G_H$ the equality holds in \eqref{inequality}. Let us show that the limit in \eqref{representation} is negative. Recall that $R_0$ is an integer larger than the range of the interaction $R$ and that $\L_n^{\oplus}$ stands for the set $\L_{n+R_0}$. We denote by $\pi_\Ln\otimes P_{\Lno\backslash \Ln}$  the law of the point process on $\Lno$ with independent configurations on $\Ln$ and $\Lno\backslash \Ln$ with distributions $\pi_\Ln$ and $P_{\Lno\backslash \Ln}$ respectively. Then

\begin{eqnarray}\label{calcul}
\lim_{n\to\infty} \frac{1}{|\Ln|}\I(P_\Ln,Q_n) & = &\lim_{n\to\infty} \frac{1}{|\Ln|}\I(P_\Lno,Q_{n+R_0})\nonumber\\  
 & = & \lim_{n\to\infty}  \frac{1}{|\Ln|}\int \ln\bigg( \frac{dP_{\Lno}}{d\pi_{\Ln}\otimes P_{\Lno\backslash \Ln}}      (\om_{\Lno}) \frac{{d\pi_{\Ln}\otimes P_{\Lno\backslash \Ln}}}{ d\pi_\Lno  }(\om_{\Lno}) \nonumber\\
 & & \qquad  \qquad \qquad \frac{d\pi_\Lno} {dQ_{n+R_0}}(\om_{\Lno}) \bigg)dP_{\Lno}(\om_{\Lno})\nonumber\\
 & = & \lim_{n\to\infty}  \frac{1}{|\Ln|}  \bigg( \I(P_{\Lno\backslash\Ln},\pi_{\Lno\backslash\Ln})  +\ln(Z_{n+R_0})+\int H(\om_{\Lno})-H_{\Lno}(\om_{\Lno})  \nonumber \\
 & &  \qquad  \qquad  -\ln(Z_\Ln(\om_\Lno)) dP_{\Lno}(\om_{\Lno}) \bigg),
 \end{eqnarray}
where the densities which appear above are given by \eqref{DLR} and \eqref{FBGibbs}. By subadditivity of the  entropy ( Proposition 15.10 in \cite{georgii}), 

$$ 0\le  \I(P_{\Lno\backslash\Ln},\pi_{\Lno\backslash\Ln}) \le \I(P_{\Lno\backslash\Ln},\pi_{\Lno\backslash\Ln})-\I(P_{\Ln},\pi_{\Ln}),$$

which implies that 

$$ \lim_{n\to \infty} \frac{1}{|\Ln|}   \I(P_{\Lno\backslash\Ln},\pi_{\Lno\backslash\Ln})=0.$$

Moreover, thanks to the existence of the mean energy and the boundary assumption assumption \eqref{boundary}, the term $\lim_{n\to\infty} |\Ln|^{-1}  \int (H(\om_{\Lno})-H_{\Lno}(\om)dP(\omega)$ vanishes as well. Therefore the limit in $\eqref{representation}$ is negative provided we show

\begin{equation}\label{inegaZ}
\liminf_{n\to \infty} \frac{1}{|\Ln|}  \int \ln\left(\frac{Z_\Ln(\om)}{Z_{n+R_0}}\right) dP(\om)\ge 0.
\end{equation}
 
From the definition of $Z_\Ln(\om)$,

$$
Z_\Ln(\om)  \ge  \int \1_{\big\{\om'_{\Ln\backslash \L_{n-R_0}}=\emptyset\big\}} e^{-H(\om'_\Ln)}\pi_\Ln(d\om_\Ln')\\
 = e^{-|\Ln\backslash \L_{n-R_0}|} Z_{n-R_0}$$

and therefore

$$ \liminf_{n\to \infty} \frac{1}{|\Ln|}  \int \ln\left(\frac{Z_\Ln(\om)}{Z_{n+R_0}}\right) dP(\om)\ge  \liminf_{n\to \infty}\frac{1}{|\Ln|} \left( \ln(Z_{n-R_0})-\ln(Z_{n+R_0})-|\Ln\backslash \L_{n-R_0}|\right)=0$$
which proves \eqref{inegaZ}. The proof of Theorem \ref{mainTh} is complete if we show that any stationary probability measure $P$ solving the equality in $\eqref{inequality}$ is a Gibbs measure. We follow essentially the scheme of \cite{B-Pre76} (In the variant used in \cite{Georgii94b}, Section 7). So let $P$ be a stationary probability measure such that $\I(P)+H(P)+p_H=0$.
Let us show that for any bounded local function $g$ and any bounded set $\Lambda$,  $\int g(\om)P(d\om)=\int g_\L(\om)P(\d\om)$ where the function $g_\L$ is defined by

$$ g_\L(\om)=\int g(\om'_\L \cup \om_{\L^c}) f_\L(\om'_\L \cup \om_{\L^c}) \pi_\L({\rm d}\om'_\L).$$

Without loss of generality we assume in the following that $|g |$ is bounded by one. Thanks to the equality \eqref{representation}, for $n$ large enough $\I(P_\Ln,Q_n)$ is finite and therefore $P_\Ln$ admits a density with respect to $Q_n$ which we denote by $f_n$. Let $\Lambda'$ be a bounded set such that $\Lambda^\oplus\subset \L'$ and such that $g$ is $\F_{\L'}$-measurable. For $n$ large enough such that $\Lambda' \subset \L_n$, the probability measure $P_{\Lambda'}$ admits a density with respect to $Q_n$ restricted to $\Lambda'$ which we denote by $f_{n,\Lambda'}$. Since $\I(P_\Ln,Q_n)/|\Ln|\to 0$, using the standard Lemma 7.5 in  \cite{Georgii94b}, for any $\delta>0$ there exists $n$ large enough and a set $\Lambda'$ with $\L^\oplus\subset \L'\subset \L_n$ such that

$$\int |f_{n,\Lambda'}-f_{n,\Lambda'\backslash \L}|dQ_n<\delta.$$

We obtain that

$$ \int g(\om)-g_\L(\om) P(d\om)= \int f_{n,\Lambda'}(\om)g(\om)-f_{n,\Lambda'\backslash \L}(\om)g_\L(\om) Q_n(d\om).$$

From the definition of $Q_n$ and since $\L^\oplus\subset \L_n$ we have

$$ \int f_{n,\Lambda'\backslash \L}(\om)g_\L(\om) Q_n(d\om)=\int f_{n,\Lambda'\backslash \L}(\om)g(\om) Q_n(d\om)$$
and we deduce that $| \int g(\om)-g_\L(\om) P(d\om)|\le \delta$. Letting $\delta$ tends to zero we get the DLR equation on $\L$. The proof of Theorem \ref{mainTh} is complete.\\

{\it Acknowledgement:} This work was supported in part by the Labex CEMPI  (ANR-11-LABX-0007-01)

\bibliography{bibaph}

\begin{thebibliography}{}

\bibitem[\protect\citeauthoryear{Baddeley and Lieshout}{Baddeley and
  Lieshout}{1995}]{A-Baddeley95}
Baddeley, A.~J. and M.~N. M.~V. Lieshout (1995).
\newblock Area-interaction point processes.
\newblock {\em Ann. Inst. Statist. Math.\/}~{\em 47\/}(4), 601--619.

\bibitem[\protect\citeauthoryear{Chiu, Stoyan, Kendall, and Mecke}{Chiu
  et~al.}{2013}]{chiu2013}
Chiu, S.~N., D.~Stoyan, W.~S. Kendall, and J.~Mecke (2013).
\newblock {\em Stochastic geometry and its applications\/} (3 ed.).
\newblock John Wiley \& Sons.

\bibitem[\protect\citeauthoryear{Dereudre, Drouilhet, and Georgii}{Dereudre
  et~al.}{2012}]{A-DerDroGeo09}
Dereudre, D., R.~Drouilhet, and H.~Georgii (2012).
\newblock Existence of {G}ibbsian point processes with geometry-dependent
  interactions.
\newblock {\em Probability Theory and Related Fields\/}~(Vol. 153, No 3-4),
  643--670.

\bibitem[\protect\citeauthoryear{Dereudre and Georgii}{Dereudre and
  Georgii}{2009}]{A-DerGeo09}
Dereudre, D. and H.~Georgii (2009).
\newblock Variational principle of gibbs measures with delaunay triangle
  interaction.
\newblock {\em Electronic Journal of Probability\/}~(14), 2438--2462.

\bibitem[\protect\citeauthoryear{Dereudre and Lavancier}{Dereudre and
  Lavancier}{2015}]{DereudreLavancier15}
Dereudre, D. and F.~Lavancier (2015).
\newblock Consistency of likelihood estimation for gibbs point processes.
\newblock {\em Preprint, arXiv:1506.02887\/}.

\bibitem[\protect\citeauthoryear{Georgii}{Georgii}{1994a}]{Georgii94b}
Georgii, H. (1994a).
\newblock The equivalence of ensembles for classical systems of particles.
\newblock {\em Journal of Statistical Physics\/}~{\em 80, Nos 5/6}, 1341--1377.

\bibitem[\protect\citeauthoryear{Georgii}{Georgii}{1994b}]{Georgii94}
Georgii, H. (1994b).
\newblock Large deviations and the equivalence of ensembles for gibbsian
  particle systems with superstable interaction.
\newblock {\em Probability Theory and Related Fields\/}~{\em 99}, 171--195.

\bibitem[\protect\citeauthoryear{Georgii}{Georgii}{1979}]{LNGeorgii}
Georgii, H.-O. (1979).
\newblock {\em Canonical Gibbs measures}.
\newblock Number 760 in Lecture Notes in Mathematics. Springer.

\bibitem[\protect\citeauthoryear{Georgii}{Georgii}{2011}]{georgii}
Georgii, H.~O. (2011).
\newblock {\em Gibbs measure and phase transitions, second edition}.
\newblock De Gruyter.

\bibitem[\protect\citeauthoryear{Kendall, Lieshout, and Baddeley}{Kendall
  et~al.}{1999}]{A-Kendall99}
Kendall, W.~S., M.~N. M.~V. Lieshout, and A.~J. Baddeley (1999).
\newblock {Q}uermass-interaction processes conditions for stability.
\newblock {\em Adv. Appli. Prob.\/}~{\em 31}, 315--342.

\bibitem[\protect\citeauthoryear{Landford and Ruelle}{Landford and
  Ruelle}{1969}]{landfordRuelle}
Landford, O. and D.~Ruelle (1969).
\newblock Observables at infinity and states with short range correlations in
  statitical mechanics.
\newblock {\em Communication in Mathematical Physics\/}~{\em 13}, 194--215.

\bibitem[\protect\citeauthoryear{Matthes, Kerstan, and Mecke}{Matthes
  et~al.}{1978}]{MKM78}
Matthes, K., J.~Kerstan, and J.~Mecke (1978).
\newblock {\em { Infinitely Divisible Point Processes }}.
\newblock Chichester: Wiley.

\bibitem[\protect\citeauthoryear{Nguyen and Zessin}{Nguyen and
  Zessin}{1979}]{A-NguZes79b}
Nguyen, X. and H.~Zessin (1979).
\newblock {Integral and differential characterizations {G}ibbs processes}.
\newblock {\em Mathematische Nachrichten\/}~{\em 88\/}(1), 105--115.

\bibitem[\protect\citeauthoryear{Preston}{Preston}{1976}]{B-Pre76}
Preston, C. (1976).
\newblock {\em Random fields}.
\newblock Springer Verlag.

\bibitem[\protect\citeauthoryear{Ruelle}{Ruelle}{1969}]{Ruelle}
Ruelle, D. (1969).
\newblock {\em Statistical Mechanics. Rigorous Results}.
\newblock (Benjamin, New-York).

\bibitem[\protect\citeauthoryear{Ruelle}{Ruelle}{1970}]{Ruelle70}
Ruelle, D. (1970).
\newblock Superstable interactions in classical statistical mechanics.
\newblock {\em Comm. Math. Phys.\/}~{\em 18}, 127--159.

\bibitem[\protect\citeauthoryear{Van~Lieshout}{Van~Lieshout}{2000}]{vanlieshout2000}
Van~Lieshout, M. (2000).
\newblock {\em Markov point processes and their applications}.
\newblock World Scientific.

\bibitem[\protect\citeauthoryear{Widom and Rowlinson}{Widom and
  Rowlinson}{1970}]{Widom70}
Widom, B. and J.~Rowlinson (1970).
\newblock New model for the study of liquid-vapor phase transitions.
\newblock {\em J. Chem. Phys.\/}~{\em 52}, 1670--1684.

\end{thebibliography}

\bibliographystyle{chicago}

\end{document}